\documentclass[12pt]{article}
\usepackage[utf8]{inputenc}
\usepackage{amssymb, amsmath,amsthm, thmtools, mathabx, mathtools}
\usepackage{bbm}
\usepackage{amsfonts}
\usepackage[left=3cm,right=3cm,top=2.5cm,bottom=2.5cm]{geometry}
\usepackage{xcolor}

\usepackage{tabto}
\TabPositions{2 cm, 4 cm, 6 cm, 8 cm}

\usepackage{tikz-cd}

\usepackage{tikz}
\usetikzlibrary{arrows, automata, positioning}

\usepackage{tocloft}
\usepackage{appendix}

\usepackage{enumerate}

\usepackage{hyperref}
\usepackage[capitalize]{cleveref}

\hypersetup{
  colorlinks   = true, 
  urlcolor     = blue, 
  linkcolor    = blue, 
  citecolor   = blue 
}
 
 \newtheorem{thmA}{Theorem}

\newtheorem{theorem}{Theorem}[section]

\newtheorem{lemma}[theorem]{Lemma}

\newtheorem*{theorem*}{Theorem}

\theoremstyle{definition}
\newtheorem{definition}[theorem]{Definition}

\newtheorem{remark}[theorem]{Remark}

\newtheorem*{remark*}{Remark}
\newtheorem*{notation*}{Notation}
\newtheorem*{acks*}{Acknowledgements}
\newtheorem*{out*}{Outline}

\theoremstyle{plain}

\renewcommand\leq{\leqslant}

\newcommand{\Z}{\mathbb{Z}}

\begin{document}

\title{An improved characterisation of \\
inner automorphisms of groups}
\author{Francesco Fournier-Facio}
\date{\today}
\maketitle

\begin{abstract}
We show that every group $G$ embeds malnormally into a simple, complete co-Hopfian group $H$. This implies that a non-trivial endomorphism of $G$ extends to $H$ if and only if it is an inner automorphism, strengthening a theorem of Schupp and answering a question of Bergman.
\end{abstract}

\section{Introduction}

Let $\alpha \colon G \to G$ be a group endomorphism, and let $\varphi \colon G \to H$ be a group homomorphism. We say that $\alpha$ \emph{extends along $\varphi$} to an endomorphism $\beta \colon H \to H$ if $\varphi \alpha = \beta \varphi$. In \cite{schupp}, answering a question of Macintyre, Schupp proved that an automorphism of $G$ is inner if and only if it extends along every homomorphism $\varphi \colon G \to H$ to an automorphism of $H$. Analogues of this result exist in other categories, e.g.\ finite groups \cite{pettet}, presheaves of groups \cite{presheaves}, groupoids \cite{groupoid} and compact groups \cite{compact}.

In \cite{bergman}, Bergman relaxed the assumption that both $\alpha$ and $\beta$ be automorphisms. He proved that inner automorphisms are exactly those non-trivial endomorphisms that extend \emph{functorially} along every homomorphism.

In this note we improve upon both results, by showing that the extension property for endomorphisms characterises non-trivial inner automorphisms, with no need for the functoriality assumed by Bergman. The non-trivial implication is \eqref{main:new} $\Rightarrow$ \eqref{main:inner}, which answers \cite[Question 5]{bergman} and \cite[Problem 21.14]{kourovka}.

\begin{thmA}
\label{main}

For every group $G$ there exists an embedding $\iota \colon G \to H$ such that, for a non-trivial endomorphism $\alpha \colon G \to G$, the following are equivalent.
\begin{enumerate}[(i)]
\item\label{main:inner} $\alpha$ is an inner automorphism.
\item\label{main:bergman} $\alpha$ extends functorially along every homomorphism $\varphi \colon G \to K$.
\item\label{main:new} $\alpha$ extends along $\iota \colon G \to H$.
\end{enumerate}
Moreover, if $G$ is countable, then $H$ can be chosen to be finitely generated.
\end{thmA}

Schupp proved his theorem by constructing, for every group $G$, an embedding $G \to H$ such that $H$ is \emph{complete} (i.e.\ it has trivial centre and outer automorphism group) and $G$ is \emph{malnormal} in $H$ (i.e.\ $hGh^{-1} \cap G = \{ 1 \}$ for all $h \in H \setminus G$). We need to also rule out maps $H \to H$ that are either not injective or not surjective, and this is done in the following theorem. Recall that a group is \emph{co-Hopfian} if every injective endomorphism is an automorphism.

\begin{thmA}
\label{embedding}

Every group $G$ embeds into a group $H$ with the following properties.
\begin{enumerate}[(1)]
\item\label{embedding:simple} $H$ is simple, complete and co-Hopfian.
\item\label{embedding:malnormal} $G$ is malnormal in $H$.
\item\label{embedding:fg} If $G$ is countable, then $H$ is finitely generated.
\end{enumerate}
\end{thmA}

\begin{proof}[Proof of \cref{main} assuming \cref{embedding}]
The implications \eqref{main:inner} $\Rightarrow$ \eqref{main:bergman} $\Rightarrow$ \eqref{main:new} are clear (and the implication \eqref{main:bergman} $\Rightarrow$ \eqref{main:inner} is \cite[Corollary 3]{bergman}), so it remains to show that \eqref{main:new} $\Rightarrow$ \eqref{main:inner}. Let $\alpha \colon G \to G$ be a non-trivial endomorphism, let $\iota \colon G \to H$ be the embedding from \cref{embedding} and suppose that $\alpha$ extends along $\iota$ to an endomorphism $\beta \colon H \to H$. Because $\alpha$ is not trivial, $\beta$ is not trivial either. Since $H$ is simple, $\beta$ is injective; since $H$ is co-Hopfian, $\beta$ is an automorphism; since $H$ is complete, $\beta$ is inner, corresponding to an element $h \in H$. Because $\alpha(G) \subset G$, it follows that $hGh^{-1} \subset G$, so $h \in G$ by \cref{embedding}\eqref{embedding:malnormal} and hence $\alpha$ is inner.
\end{proof}

The rest of this note is devoted to the proof of \cref{embedding}. Ensuring that $H$ is co-Hopfian is the main new difficulty. When $G$ is countable and avoids some order of torsion, Miller and Schupp embed $G$ into a (possibly not simple) complete group that is Hopfian and co-Hopfian \cite{miller:schupp} (see also \cite{bridson:short}). Minasyan observed \cite{overflow} that it follows from an embedding theorem of Olshanskii \cite{olshanskii} that every countable group embeds into a finitely generated simple co-Hopfian group; his idea applies almost directly to prove \cref{embedding} in the countable case, using a stronger embedding theorem of Obraztsov \cite{obraztsov}.

However when $G$ is uncountable, \cref{embedding} is false if one insists that $G$ and $H$ have the same cardinality (\cref{necessary}), as they do in all the embeddings mentioned so far. To work around this, we use Corson's construction of J{\'o}nsson groups of large cardinality \cite{corson} as an additional ingredient.

\begin{acks*}
The author is supported by the Herchel Smith Postdoctoral Fellowship Fund. He thanks George Bergman, Sam Corson and Ashot Minasyan for useful conversations.
\end{acks*}

\section{The countable case}

We start by assuming that $G$ is countable. In this case we can directly apply an embedding theorem of Obraztsov \cite{obraztsov} (which is a refinement of several previous versions \cite{obraztsov0, obraztsov1, obraztsov2, olshanskii}).

\begin{theorem}[Obraztsov {\cite[Theorem C]{obraztsov}}]
\label{cobr}
Let $G_1, G_2$ be two non-trivial countable groups, not both of order $2$. Then $G_1, G_2$ embed into a group $H$ with the following properties:
\begin{enumerate}[(1)]
\item\label{cobr:2gen} If $g \in G_i, h \notin G_i$ and $g^2 \neq 1$, then $\langle g, h \rangle = H$;
\item\label{cobr:simple} $H$ is simple and complete;
\item\label{cobr:malnormal} $G_i$ is malnormal in $H$;
\item\label{cobr:subgroups} Every proper subgroup of $H$ is either infinite cyclic, or infinite dihedral, or conjugate to a subgroup of a $G_i$.
\end{enumerate}
\end{theorem}

\begin{proof}[About the proof]
The only claim that is not explicitly stated in \cite{obraztsov} is the malnormality: rather, it is only stated that the normaliser of $G_i$ in $H$ is contained in $G_i$ \cite[Theorem A(11)]{obraztsov}. However, the author establishes this fact in \cite[end of p.235]{obraztsov} simply referring to \cite[Lemma 34.10]{olshanskii:book}, which in fact shows malnormality. Note moreover that malnormality is explicitly stated in \cite[Theorem 2(4)]{olshanskii}, that \cite{obraztsov} generalises.
\end{proof}

\begin{proof}[Proof of \cref{embedding} in case $G$ is countable]
Let $G_1 \coloneqq G$. We first claim that there exists a non-trivial countable group $G_2$ without involutions such that neither $G_1$ nor $G_2$ embeds in the other. If $G_1$ is not torsion-free, then we can take $G_2$ to be a finitely generated torsion-free group that does not embed into $G_1$ (this exists since a countable group has only countably many finitely generated subgroups, and there exist uncountably many isomorphism classes of finitely generated groups \cite{neumann}). If $G_1$ is torsion-free, then we can take $G_2$ to be $\Z/3$.

Let $H$ be the group given by \cref{cobr} for this choice of $G_1$ and $G_2$. By \cref{cobr}\eqref{cobr:2gen}, it is generated by a non-trivial element in $G_2$ and an element in $G_1$, giving \ref{embedding}\eqref{embedding:fg}. \cref{cobr}\eqref{cobr:malnormal} gives \ref{embedding}\eqref{embedding:malnormal}. To complete the proof of \ref{embedding}\eqref{embedding:simple}, we need to show that $H$ is co-Hopfian, so let $\beta \colon H \to H$ be an injective endomorphism of $H$. If $\beta(H)$ were a proper subgroup of $H$, since it is not infinite cyclic or infinite dihedral, it would have to be conjugate to a subgroup of a $G_i$ by \cref{cobr}\eqref{cobr:subgroups}. But $\beta(H) \cong H$, and so this contradicts the assumption that $G_1$ and $G_2$ do not embed into each other. Therefore $\beta$ is an automorphism, which concludes the proof.
\end{proof}

\begin{remark}
\label{fp}

While $H$ can be chosen to be finitely generated, it may not be possible to choose it to be finitely presented, even if $G$ is. Indeed, there exists a finitely presented group $G$ that contains a copy of every finitely presented group \cite{higman}. Let $G \to H$ be an embedding, where $H$ is finitely presented. Then there exist embeddings $H \to H \times \Z \to G \to H$, whose composition is an endomorphism of $H$ that is injective but not surjective, so $H$ is not co-Hopfian.
\end{remark}

\section{The uncountable case}

Now suppose that $G$ is uncountable. Our first ingredient is again an embedding theorem of Obraztsov \cite{obraztsov}. We use the convention that a set is \emph{countable} if it has cardinality at most $\aleph_0$.

\begin{definition}
\label{def:gen}
Let $G_1, G_2$ be two non-trivial groups, and let $\Omega \coloneqq (G_1 \setminus \{ 1 \}) \sqcup (G_2 \setminus \{ 1 \})$. Let $\mathcal{P}'(\Omega) \coloneqq \{ A \subset \Omega : A \neq \emptyset \}$. A \emph{generating function} is a function $f \colon \mathcal{P}'(\Omega) \to \mathcal{P}'(\Omega)$ with the following properties.
\begin{enumerate}[(a)]
\item\label{gen:basic} If $A \subset G_i$ then $f(A) = \langle A \rangle \setminus \{ 1 \} \subset G_i \setminus \{ 1 \} \subset \Omega$.
\item\label{gen:dihedral} If $A \nsubset G_i$ for $i = 1, 2$ and $A$ consists of two involutions (we say that $A$ is \emph{dihedral}), then $f(A) = A$.
\item\label{gen:finite} If $A \nsubset G_i$ for $i = 1, 2$ and $A$ is finite and not dihedral, then $B \coloneqq f(A)$ is countable, contains $A$, and contains $f(C)$ for any finite $C \subset B$.
\item\label{gen:infinite} If $A \subset \Omega$ is infinite, then $f(A)$ is the union of the sets $f(C)$, where $C$ ranges over all finite subsets of $A$.
\end{enumerate}
\end{definition}

\begin{theorem}[Obraztsov {\cite[Theorem A]{obraztsov}}]
\label{obr}
Let $G_1, G_2$ be two non-trivial groups, not both of order $2$, let $\Omega \coloneqq (G_1 \setminus \{ 1 \}) \sqcup (G_2 \setminus \{ 1 \})$, and let $f \colon \mathcal{P}'(\Omega) \to \mathcal{P}'(\Omega)$ be a generating function. Then $G_1$ and $G_2$ embed into a group $H$ with the following properties:
\begin{enumerate}[(1)]
\item\label{obr:gen} $H = \langle G_1, G_2 \rangle$;
\item\label{obr:simple} $H$ is simple and complete;
\item\label{obr:malnormal} $G_i$ is malnormal in $H$; 
\item\label{obr:subgroups} If $K \subset H$ is a proper subgroup, then one of the following holds:
\begin{itemize}
\item $K$ is infinite cyclic or infinite dihedral;
\item $K$ is conjugate to a subgroup of a $G_i$;
\item There exists $A \subset \Omega$ with $A \nsubset G_i$ for $i = 1, 2$ such that $K$ is conjugate to $\langle A \rangle$, which contains $f(A)$.
\end{itemize}
\end{enumerate}
\end{theorem}

The same comment on malnormality as in \cref{cobr} applies. Our second ingredient comes from J{\'o}nsson groups, particularly Corson's construction \cite{corson}.

\begin{definition}
A group is \emph{J{\'o}nsson} if every proper subgroup has strictly smaller cardinality.
\end{definition}

Examples of J{\'o}nsson groups include finite groups, the Pr{\"u}fer $p$-groups and Tarski monster $p$-groups \cite{tarski}. The first uncountable J{\'o}nsson groups were constructed by Shelah \cite{shelah}, and recently Corson showed that there exist J{\'o}nsson groups of arbitrarily large cardinality \cite{corson}, using an earlier embedding theorem of Obraztsov \cite{obraztsov2}. More precisely, he proved that, for every J{\'o}nsson algebra, there exists a J{\'o}nsson group of the same cardinality \cite[Theorem 1]{corson}. Therefore using \cite{cardinal}, there exist J{\'o}nsson groups of cardinality $\kappa$, whenever $\kappa$ is the successor of a regular cardinal. We will use a similar idea as in his construction.

\begin{theorem}[Corson]
\label{corson}

Let $\kappa$ be the successor of a regular cardinal (e.g.\ $\kappa = \aleph_2$). Then there exists a torsion-free J{\'o}nsson group of cardinality $\kappa$. Moreover, for every odd prime $p$, there exists a J{\'o}nsson group of cardinality $\kappa$ with the property that every non-cyclic subgroup contains an element of order $p$.
\end{theorem}

\begin{proof}[About the proof]
The first statement is given by \cite[Theorem 1.1]{corson}, combined with \cite{cardinal}. The second statement follows with the same proof, taking cyclic groups of order $p$ as input, instead of infinite cyclic groups. We insist on $p$ being odd for coherence with the presentation of \cite{corson}, which appeals to the earlier embedding theorem of Obraztsov \cite{obraztsov2}, which requires that all groups involved have no involutions; although $p = 2$ could also be allowed by appealing instead to the stronger embedding theorem that we are using here \cite{obraztsov}.
\end{proof}

We also need the following lemma. Recall that a \emph{magma} $\Omega$ is a set with a binary operation $\cdot_\Omega \colon \Omega \times \Omega \to \Omega$, and an \emph{identity} in $\Omega$ is a (not necessarily unique) element $1$ such that $x \cdot_\Omega 1 = 1 \cdot_\Omega x = x$ for all $x \in \Omega$.

\begin{lemma}
\label{algebras}
Let $\Omega_1, \Omega_2$ be two magmas with identity $1$, and with $|\Omega_1| \leq |\Omega_2|$. Let $\Omega'$ be their disjoint union, glued along the identity $1$. Then there exists a magma structure on $\Omega'$ extending that of the $\Omega_i$, with the property that $\Omega_2$ is a maximal submagma of $\Omega'$.
\end{lemma}

\begin{proof}
Let $\pi \colon \Omega_2 \to \Omega_1$ be a surjection. For $x, y \in \Omega'$, if $x, y \in \Omega_i$ define $x \cdot_{\Omega'} y \coloneqq x \cdot_{\Omega_i} y$. If $x \in \Omega_1 \setminus \{1\}$ and $y \in \Omega_2 \setminus \{ 1 \}$, define $x \cdot_{\Omega'} y = y \cdot_{\Omega'} x \coloneqq \pi(y)$.

Now suppose that $x \in \Omega_1 \setminus \{ 1 \}$. Then $\langle x, \Omega_2 \rangle_{\Omega'}$ contains $\Omega_2$ as well as $x \cdot_{\Omega'} y = \pi(y)$, for each $y \in \Omega_2 \setminus \{1\}$, hence it contains $\Omega_1$ as well.
\end{proof}

We are finally ready to conclude the proof of \cref{embedding}.

\begin{proof}[Proof of \cref{embedding} in case $G$ is uncountable]
Let $G_1 \coloneqq G$. Let $\kappa$ be the successor of a regular cardinal such that $\kappa > |G_1|$, and let $G_2$ be a J{\'o}nsson group of cardinality $\kappa$ (which therefore does not embed into $G_1$) such that $G_1$ does not embed into $G_2$. This exists by \cref{corson}: if $G_1$ is not torsion-free, we take $G_2$ to be torsion-free, and if $G_1$ is torsion-free, we take $G_2$ to be such that every non-cyclic subgroup contains an element of order $3$. Let $\Omega \coloneqq (G_1 \setminus \{ 1 \}) \sqcup (G_2 \setminus \{ 1 \})$ and let $\Omega' \coloneqq \Omega \sqcup \{ 1 \}$.

By \cref{algebras} we may endow $\Omega'$ with a magma structure extending the group structures of $G_1$ and $G_2$, such that $G_2$ is maximal in $\Omega'$. For $A \subset \Omega'$ we let $\langle A \rangle_{\Omega'}$ denote the submagma of $\Omega'$ generated by $A$, and $\langle A \rangle_\Omega \coloneqq \langle A \rangle_{\Omega'} \setminus \{ 1 \}$. In particular, for $A \subset G_i$ we have $\langle A \rangle_{G_i} = \langle A \rangle_{\Omega'}$. We define a function $f \colon \mathcal{P}'(\Omega) \to \mathcal{P}'(\Omega)$ as follows.
\begin{itemize}
\item If $A \nsubset G_i$ for $i = 1, 2$ and $A$ is dihedral, then $f(A) = A$.
\item Otherwise, $f(A) = \langle A \rangle_\Omega$.
\end{itemize}

We claim that $f$ is a generating function (\cref{def:gen}). If $A \subset G_i$ then $f(A) = \langle A \rangle_\Omega = \langle A \rangle_{G_i} \setminus \{1\}$, giving condition \eqref{gen:basic}. Condition \eqref{gen:dihedral} is clear. Now assume that $A \nsubset G_i$ for $i = 1, 2$ and $A$ is not dihedral. Suppose first that $A$ is finite. Then $B \coloneqq f(A) = \langle A \rangle_\Omega$ is countable, because $\langle A \rangle_{\Omega'}$ is a finitely generated magma. Moreover, if $C \subset B$ is finite, then either $f(C) = C \subset B$, or $f(C) = \langle C \rangle_\Omega \subset \langle B \rangle_\Omega = B$, giving condition \eqref{gen:finite}. Suppose next that $A$ is infinite. Then $f(A) = \langle A \rangle_\Omega$, which is the union of $\langle C \rangle_\Omega$ over all finite subsets $C \subset A$, equivalently over all finite subsets $C \subset A$ of size at least $3$, for which $\langle C \rangle_\Omega = f(C)$, giving condition \eqref{gen:infinite}.

Now apply \cref{obr} to this data, and let $H$ be the resulting group. \cref{obr}\eqref{obr:malnormal} gives \ref{embedding}\eqref{embedding:malnormal}. To complete the proof of \ref{embedding}\eqref{embedding:simple}, we need to show that $H$ is co-Hopfian, so let $\beta \colon H \to H$ be an injective endomorphism. Because $G_1$ and $G_2$ do not embed into each other, $\beta(H) \cong H$ is not conjugate to a subgroup of either $G_i$, and of course it is not infinite cyclic or infinite dihedral. Hence up to conjugacy we may assume that $\beta(H) = \langle A \rangle_H \supset f(A)$ for some $A \subset \Omega$ such that $A \nsubset G_i$ for $i = 1, 2$. Because $\kappa = |\beta(H)| = |\langle A \rangle_H |$, necessarily $|A| = \kappa$. Hence writing $A_i \coloneqq A \cap G_i$, since $\kappa > |G_1|$, we must have $|A_2| = \kappa$. Since $G_2$ is J{\'o}nsson, $\langle A_2 \rangle_{\Omega'} = \langle A_2 \rangle_{G_2} = G_2,$ so $f(A) = \langle A \rangle_\Omega$ contains $G_2 \setminus \{ 1 \}$. Because $G_2$ is a maximal submagma of $\Omega'$ and $A_1 \neq \emptyset$, we get $\langle A \rangle_{\Omega'} = \Omega'$, so $f(A) = \Omega$. Therefore $\beta(H) = H$ by \cref{obr}\eqref{obr:gen}, showing that $\beta$ is an automorphism, which concludes the proof.
\end{proof}

\begin{remark}
\label{necessary}
Unlike in the countable case, in this argument the group $H$ from \cref{embedding} has cardinality strictly greater than that of $G$. For example, if $G$ has cardinality $\aleph_1$, we can choose $H$ to have cardinality $\aleph_2$.

This is necessary in general. Indeed, for certain uncountable cardinals $\kappa$ \cite{kegel} (and even for all uncountable cardinals, assuming the generalised continuum hypothesis \cite[Proposition 2]{gch}), there exists a group $G$ of cardinality $\kappa$ that contains a copy of every group of cardinality $\kappa$. Such a group cannot embed into a co-Hopfian group of cardinality $\kappa$, by the same argument as in \cref{fp}.
\end{remark}

\footnotesize

\bibliographystyle{amsalpha}
\bibliography{ref}

\vspace{0.5cm}

\normalsize

\noindent{\textsc{Department of Pure Mathematics and Mathematical Statistics, University of Cambridge, UK}}

\noindent{\textit{E-mail address:} \texttt{ff373@cam.ac.uk}}

\end{document}